\documentclass[a4paper,12pt]{amsart}

\usepackage{amsmath}
\usepackage{amssymb}
\usepackage{mathrsfs}
\usepackage{enumerate}
\usepackage{ifthen}
\usepackage{graphicx}
\usepackage{caption}
\usepackage{subcaption}
\usepackage{float}
\usepackage[T1]{fontenc} 

\nonstopmode \numberwithin{equation}{section}
\newtheorem{thm}{Theorem}[section]

\newtheorem{lem}{Lemma}[section]

\theoremstyle{definition}
\newtheorem{definition}{Definition}[section]
\newtheorem{example}{Example}[section]

\newtheorem{rem}{Remark}[section]

\newcounter{minutes}
\divide\time by 60
\newcounter{hours}
\multiply\time by 60
\addtocounter{minutes}{-\time}

\newcounter {own}
\def\theown {\thesection       .\arabic{own}}

{\qed\bigskip}

\newcounter{alphabet}

\begin{document}
	
	\title{On the pre-Schwarzian norm and starlikeness of certain logharmonic mappings}

	\author{Sushil Pandit}
	\address{Sushil Pandit,
		Department of Mathematics,
		National Institute of Technology Durgapur,
		Near Mahatma Gandhi Avenue,
		Durgapur-713209, West Bengal, India.}
	\email{sushilpandit15594@gmail.com}

	\subjclass[2010]{Primary 30C45, 30C55}
	\keywords{logharmonic mapping; analytic functions; pre-Schwarzian norm; growth theorem; logharmonic hereditarily starlike functions.}
	
	\def\thefootnote{}
	\footnotetext{ {\tiny File:~\jobname.tex,
			printed: \number\year-\number\month-\number\day,
			          \thehours.\ifnum\theminutes<10{0}\fi\theminutes }
		} \makeatletter\def\thefootnote{\@arabic\c@footnote}\makeatother

	\begin{abstract}
		In this note, we consider certain logharmonic mappings in the unit disk $\mathbb{D}=\{z\in\mathbb{C}:|z|<1\}.$ Next, we obtain sharp bound of pre-Schwarzian norm of such logharmonic mappings in the unit disk. Then we discuss growth theorem for the mappings. Moreover, we discuss starlikeness of logharmonic mappings and compute sufficient coefficient condition of hereditarily starlikeness. At the end, we present some example of logharmonic hereditarily starlike function.
	\end{abstract}
	
	\thanks{}
	
	\maketitle
	\pagestyle{myheadings}
	\markboth{Sushil Pandit}{On the pre-Schwarzian norm and starlikeness of certain logharmonic mappings}

	\section{Introduction}

	Let $\mathcal{S}$ be the class of all analytic and univalent functions $h$ defined in the unit disk $\mathbb{D}=\{z\in\mathbb{C}:|z|<1\}$ with normalization $h(0)=h'(0)-1=0.$ Let $\mathcal{R}$ be class of functions $h\in\mathcal{S}$ such that ${\rm Re\,}h'(z)>0$ for $z\in\mathbb{D}.$ In 1962, Macgregor \cite{Macgregor-1962} studied the class $\mathcal{R}$. It is well known that every analytic function $h\in\mathcal{R}$ is close-to-convex. A complex valued mapping $f$ is called logharmonic if it is a solution of the nonlinear elliptic partial differential equation
	\begin{align}\label{log-000020}
		\frac{\overline{f_{\overline{z}}}}{\overline{f}}=\omega\frac{f_z}{f}
	\end{align}
	where $\omega$ is analytic self mapping of the unit disk $\mathbb{D}$ which is also known as second dilatation or simply dilatation of $f.$ The Jacobian of $f$ is given by $J_f=|f_z|^2-|f_{\overline{z}}|^2=|f_z|^2(1-|\omega|^2).$ We note that the dilatation $\omega$ satisfies $|\omega|<1$ and so $J_f>0.$ Thus every non constant logharmonic mapping $f$ is always sense-preserving. If $f$ is non-constant logharmonic mapping in $\mathbb{D}$ and vanishes only at $z = 0,$ then $f$ has the representation
	\begin{align*}
		f(z) = z^m|z|^{2\beta m}h(z)\overline{g(z)},
	\end{align*}
	where $m$ is non-negative integer, ${\rm Re\,}(\beta)>-1/2$ and $h$ and $g$ are analytic functions in $\mathbb{D}$ such that $g(0)=1$ and $h(0)\neq 0$ (see \cite{Abdulhadi-1988,Abdulhadi-Ali-2012}). Authors in \cite{Abdulhadi-1988} have shown that if $f$ is a nonvanishing logharmonic mapping, then it can be expressed as
	\begin{align}\label{log-000040}
		f(z) = h(z)\overline{g(z)},
	\end{align}
	where $h$ and $g$ are nonvanishing analytic functions in the unit disk $\mathbb{D}.$ Further, if the mapping $f$ given by \eqref{log-000040} is locally univalent and sense-preserving, then $h'g \neq 0$ in $\mathbb{D}$ and the dilatation $\omega$ is given by
	\begin{align}\label{p2-010}
		\omega=(g'h)/(gh').
	\end{align}
	There are several fundamental results on logharmonic mappings defined on the unit disk $\mathbb{D}$ (see \cite{Abdulhadi-Ali-2012, Liu-Ponnusamy-2018}). It is easy to see that if $f=h\overline{g}$ is a nonvanishing logharmonic mapping defined in $\mathbb{D}$, then the function $\log{f}=
	\log{h}+\overline{\log{g}}$ is harmonic in $\mathbb{D},$ and its corresponding dilatations are the same. A twice continuously differentiable complex valued function $f$ in a domain $\Omega$ is called harmonic if it satisfies the Laplace equation $\Delta f = 4f_{z\overline{z}} = 0.$ The quantity $\omega=\overline{f_{\overline{z}}}/f_{z}$ is called dilatation of the harmonic mapping $f$. In a simply connected domain $\Omega,$ every harmonic mapping $f$ has a canonical representation of the form $f = h+\overline{g},$ where $h$ and $g$ are analytic functions in  $\Omega$ called the analytic and co-analytic part of $f$ respectively. Let $\mathcal{H}$ denotes the class of locally univalent sense-preserving harmonic mappings $f=h+\overline{g}$ in the unit disk $\mathbb{D}$ such that $h(0)=h'(0)-1=g(0)=0.$ This kind of mappings have been well discussed in \cite{Clunie-Small-1984, Duren-2004}.

	\subsection{Bloch Functions:}
An analytic function $h$  defined in $\mathbb{D}$ is called a Bloch function (see \cite{Pommerenke-1970}) if $\beta_h=\sup_{z\in\mathbb{D}}(1-|z|^2)|h'(z)|<\infty.$	The collection $\mathcal{B}$ of analytic Bloch functions in $\mathbb{D}$ form a Banach space with the norm given by
	\begin{align*}
		||h||_{\mathcal{B}}=|h(0)|+\sup_{z\in\mathbb{D}}(1-|z|^2)|h'(z)|.
	\end{align*}
	A harmonic mapping $f=h+\overline{g}\in\mathcal{H}$ is called harmonic Bloch if	$\beta_f=\sup_{z\in\mathbb{D}}(1-|z|^2)(|h'(z)|+|g'(z)|)<\infty.$
	For more information about harmonic Bloch mapping, we refer \cite{Chen-Gauthier-Hengartner-2000, Colonna-1989}. Similarly, a nonvanishing logharmonic mapping $f(z) = h(z)\overline{g(z)}$ in $\mathbb{D}$ is called logharmonic Bloch if
	\begin{align}\label{p2-015}
		\beta_f=\sup_{z\in\mathbb{D}}(1-|z|^2)\left\{\left|\frac{h'(z)}{h(z)}\right|+\left|\frac{g'(z)}{g(z)}\right|\right\}<\infty.
	\end{align}
	The space $\mathcal{B}_{Lh}$ of all logharmonic Bloch functions forms a complex Banach space with the norm given by $	||f||_{\mathcal{B}_{Lh}}=|f(0)|+\beta_f$ (see \cite{Liu-Ponnusamy-2018}).

	\subsection{Pre-Schwarzian norm:}
	For a locally univalent analytic function $h$ defined in $\mathbb{D}$, the pre-Schwarzian norm $\|P_h\|$ is defined by
	\begin{align}\label{p2-025}
		||P_h|| = \sup_{z \in \mathbb{D}}(1-|z|^2)|P_h(z)|,
	\end{align}
	where $P_h(z)=\frac{h''(z)}{h'(z)}$	is the pre-Schwarzian derivative of $h$. Various valuable univalence criteria for an analytic function $h$ were obtained using pre-Schwarzian derivative of $h$. It has been shown in \cite{Kruas-1932} that $||P_h||\leq 6$ for a univalent function $h.$ Conversely, for a locally univalent function $h$ if $||P_h||\leq 1$, then the function $h$ is univalent in $\mathbb{D}$ (see \cite{Becker-1972}, \cite{Becker-Pommerenke-1984}). In 1976, Yamashita \cite{Yamashita-1976} proved that $||P_h||$ is finite if and only if $h$ is uniformly locally univalent in $\mathbb{D},$ that is, there exists a constant $\rho>0$ such that $h$ is univalent on the hyperbolic disk $|(z-b)/(1-\overline{b}z)|<\tanh\rho$ of radius $\rho$ for every $b\in\mathbb{D}.$ In 2015, Hern{\'a}ndez and Mart{\'i}n \cite{Hernandez-Martin-2015} proposed the definitions of the pre-Schwarzian derivative and the Schwarzian derivative of a locally univalent harmonic mapping of the form $f=h+\overline{g}.$ For more information about the pre-Schwarzian and Schwarzian derivatives of harmonic mappings, we refer to the articles, \cite{Hernandez-Martin-2015, Chuaqui-Hernandez-Martin-2017, Liu-Ponnusamy--2018, Ali-Pandit-2023, Wang-Li-Fan-2024}.

	For a locally univalent logharmonic mapping $f$ of the form \eqref{log-000040}, Bravo et al. \cite{Bravo-Hernandez-Ponnusamy-Venegas-2022} defined the pre-Schwarzian derivative $P_f$ as
	\begin{align}\label{p2-040}
		P_f=\left(\log(J_f)\right)_z =\frac{h''}{h'}+\frac{g'}{g}-\frac{\overline{\omega}\omega'}{1-|\omega|^2}.
	\end{align}
	where $J_f$ is the  Jacobian and $\omega$ is the dilatation of the function $f.$ Further, the pre-Schwarzian norm $||P_f||$ is defined by 
	$$||P_f|| = \sup\limits_{z \in \mathbb{D}}(1-|z|^2)|P_f(z)|.$$
	If $f$ is a sense-preserving logharmonic mapping of the form \eqref{log-000040} and $\phi$ is a locally univalent analytic function for which the composition $f\circ\phi$ is well defined, then the function $f\circ\phi$ is again a sense-preserving logharmonic mapping and the pre-Schwarzian derivative of it is given by
	\begin{align*}
		P_{f\circ\phi}=P_f(\phi)\phi'+P_\phi.
	\end{align*}
	Authors in \cite{Bravo-Hernandez-Ponnusamy-Venegas-2022} have studied various properties of the pre-Schwarzian norm of a sense-preserving logharmonic mapping. In this article, we obtain pre-Schwarzian norm estimate of certain logharmonic mappings of the form \eqref{log-000040}.\\

	\section{The Class $L_{\mathcal{R}}$}
	In this section, we introduce a class $L_{\mathcal{R}}$ of locally univalent logharmonic mappings $f=H\overline{G}\in C^1(\mathbb{D})$ (i.e., continuously differentiable functions on $\mathbb{D}$), where $H=e^h$ such that $h\in\mathcal{R},$ having dilatation $\omega=(G'H)/(GH')=G'/(Gh'):\mathbb{D}\rightarrow\mathbb{D}.$ Here, we note that every $f\in L_{\mathcal{R}}$ is non vanishing in $\mathbb{D}.$ It is quite clear that if $f=e^h\overline{G}\in L_{\mathcal{R}}$ then there exists some analytic self mapping $\epsilon$ of $\mathbb{D}$ with $\epsilon(0)=0$ such that
	$$h'=\frac{1+\epsilon}{1-\epsilon}.$$
	From this it follows that
	\begin{align*}
		\beta_f= & \sup\limits_{z\in\mathbb{D}}(1-|z|^2)\left[|h'(z)|+\left|\frac{G'(z)}{G(z)}\right|\right]\\
		=  & \sup\limits_{z\in\mathbb{D}}(1-|z|^2)(1+|\omega(z)|)|h'(z)|\\
		\le & 2\sup\limits_{z\in\mathbb{D}}(1-|z|^2)\frac{1+|z|}{1-|z|}\\
		= & 8,
	\end{align*}
	that is, $\beta_f$ is finite and so $f$ is logharmonic Bloch mapping. Moreover, for $f=e^{h}\overline{G}\in L_{\mathcal{R}}$ having dilatation $\omega=G'/(Gh')$, the function $\log{f}=h+\overline{\log{G}}$ is harmonic as well as harmonic Bloch with same dilatation $\omega$ because
	$$\beta_{\log{f}}=\sup\limits_{z\in\mathbb{D}}(1-|z|^2)\left|(1+\omega(z))h'(z)\right|=\beta_f.$$

	It is well known that the pre-Schwarzian norm of an univalent analytic or an univalent harmonic mapping  is finite. The same is not true in the case of logharmonic mapping. For example, the logharmonic koebe mapping $K=H\overline{G}$ where
	\begin{align*}
		H(z)=\frac{z}{1-z}\exp{\left(\frac{2z}{1-z}\right)}\quad\text{and}\quad G(z)=(1-z)\exp{\left(\frac{2z}{1-z}\right)}
	\end{align*}
	with dilatation $z$ is univalent but does not have finite $\|P_K\|$ (see \cite{Bravo-Hernandez-Ponnusamy-Venegas-2022}). Thus there arises a natural question that under what condition an univalent logharmonic mapping has finite pre-Schwarzian norm. Ali and Pandit in \cite{Ali-Pandit-2024} obtained various necessary and sufficient conditions of the finiteness of pre-Schwarzian norm of logharmonic mappings of the form \eqref{log-000040}. 
	Here, we observe that every $f=e^{h}\overline{G}\in L_{\mathcal{R}}$ is logharmonic Bloch mapping and the pre-Schwarzian norm of $e^{h},$ where $h\in\mathcal{R},$ is finite. Thus from \cite[Theorem 3.4]{Ali-Pandit-2024} it follows that $f$ has finite pre-Schwarzian norm. Our following result provides exact bound of the pre-Schwarzian norm $\|P_f\|$ for $f\in L_{\mathcal{R}}.$

	\begin{thm}\label{log-000100}
		Let $f\in L_{\mathcal{R}}$ be a sense-preserving logharmonic mapping in the unit disk $\mathbb{D}.$ Then the pre-Schwarzian norm $\|P_f\|\le 11.$ The estimate is sharp.
	\end{thm}
	\begin{proof}
		The pre-Schwarzian derivative of a logharmonic mapping $f\in L_{\mathcal{R}}$ of the form $H(z)\overline{G(z)}=e^{h(z)}\overline{G(z)}$ with dilatation $\omega=G'/(Gh')$ is
		\begin{align*}
			P_f(z)=P_H(z)+\omega(z)\frac{H'(z)}{H(z)}-\frac{\overline{\omega(z)}\omega'(z)}{1-|\omega(z)|^2}.
		\end{align*}
		From hypothesis, it follows that $H(z)=e^{h(z)}$ where $h\in\mathcal{R}.$ Thus for some analytic function $\epsilon:\mathbb{D}\rightarrow\mathbb{D}$ with $\epsilon(0)=0,$ we get
		\begin{align}\label{log-000110}
			\frac{H'(z}{H(z)}=h'(z)=\frac{1+\epsilon(z)}{1-\epsilon(z)},\quad P_h(z)=\frac{2\epsilon'(z)}{1-\epsilon^2(z)}
		\end{align}
		and so
		\begin{align*}
			P_H(z)=h'(z)+P_h(z)=\frac{1+\epsilon(z)}{1-\epsilon(z)}+\frac{2\epsilon'(z)}{1-\epsilon^2(z)}.
		\end{align*}
		Therefore, $$P_f=\frac{2\epsilon'(z)}{1-\epsilon^2(z)}+\left(1+\omega(z)\right)\frac{1+\epsilon(z)}{1-\epsilon(z)}-\frac{\overline{\omega(z)}\omega'(z)}{1-|\omega(z)|^2}.$$
		Next by Schwarz Pick lemma, the pre-Schwarzian norm of $f$ is
		\begin{align*}
			\|P_f\|= & \sup\limits_{z\in\mathbb{D}}|P_f(z)|(1-|z|^2)\\ = 
			& \sup\limits_{z\in\mathbb{D}}\left|\frac{2\epsilon'(z)}{1-\epsilon^2(z)}+\left(1+\omega(z)\right)\frac{1+\epsilon(z)}{1-\epsilon(z)}-\frac{\overline{\omega(z)}\omega'(z)}{1-|\omega(z)|^2}\right|(1-|z|^2)\\
			\le & \sup\limits_{z\in\mathbb{D}}\left[2+2(1+|z|)^2+1\right]\\
			= & 11.
		\end{align*}
		To show that the estimate is sharp, we consider the function $f_t(z)=H(z)\overline{G(z)}=e^{h(z)}\overline{G(z)}$ where $h(z)=-z-2\log{(1-z)}$ and dilatation
		$$\omega_t(z)=\frac{t-z}{1-tz},~~t\in(0,1).$$
		It is easy to calculate that
		\begin{align}\label{log-000120}
			h'(z)=\frac{1+z}{1-z},\quad P_h(z)=\frac{2}{1-z^2}
		\end{align}
		and
		\begin{align}\label{log-000130}
			\frac{\overline{\omega_t(z)}\omega_t'(z)}{1-|\omega_t(z)|^2}=\frac{\overline{z}-t}{(1-tz)(1-|z|^2)},\quad 1+\omega_t(z)=\frac{(1+t)(1-z)}{1-tz}.
		\end{align}
		Since $f_t\in L_{\mathcal{R}}$ from Theorem \ref{log-000100} it follows that the pre-Schwarzian norm $\|P_{f_t}\|\le 11.$ We need to show that $\|P_{f_t}\|=11.$ From the hypothesis
		\begin{align}\label{log-000146}
			\|P_{f_t}\|= & \sup\limits_{z\in\mathbb{D}}\left|P_{f_t}\right|(1-|z|^2)\\\nonumber
			= & \sup\limits_{z\in\mathbb{D}}\left|P_H(z)+\omega_t(z)\frac{H'(z)}{H(z)}-\frac{\overline{\omega_t(z)}\omega_t'(z)}{1-|\omega_t(z)|^2}\right|(1-|z|^2)\\\nonumber
			= & \sup\limits_{z\in\mathbb{D}}\left|P_h(z)+\left(1+\omega_t(z)\right)h'(z)-\frac{\overline{\omega_t(z)}\omega_t'(z)}{1-|\omega_t(z)|^2}\right|(1-|z|^2)\\\nonumber
			= & \sup\limits_{z\in\mathbb{D}}\left|\frac{2}{1-z^2}+\frac{(1+t)(1+z)}{1-tz}+\frac{t-\overline{z}}{(1-tz)(1-|z|^2)}\right|(1-|z|^2).
		\end{align}
		Now the value of 
		\begin{align}\label{log-000150}
			N_t= & \sup\limits_{z\in(0,1)}\left|\frac{2}{1-z^2}+\frac{(1+t)(1+z)}{1-tz}+\frac{t-\overline{z}}{(1-tz)(1-|z|^2)}\right|(1-|z|^2)\\\nonumber        
			= &\sup\limits_{r\in(0,1)}\left|2+\frac{(1+t)(1+r)(1-r^2)}{1-tr}+\frac{t-r}{1-tr}\right|\\\nonumber
			= & \sup\limits_{0\le r<1}\left|2+\frac{(1+t)(1+r)(1-r^2)+(t-r)}{1-tr}\right|\\\nonumber
			= & \sup\limits_{0\le r<1}E(r,t)\le 11
		\end{align}
		where $$E(r,t)=\left|2+\frac{(1+t)(1+r)(1-r^2)+(t-r)}{1-tr}\right|$$ from which it is simple to find that
		$\lim\limits_{r\rightarrow 1}\lim\limits_{t\rightarrow 1} E(r,t)=11.$ Taking the limit $t$ tends to $1$ in the relation \eqref{log-000150} we get
		\begin{align}\label{e5}
			\lim\limits_{t\rightarrow 1}N_t=\lim\limits_{t\rightarrow 1}\sup\limits_{0\le r<1} E(r,t)\le 11.
		\end{align}
		For $t\in(0,1)$ it is clear that
		$$E(r,t)\le \sup\limits_{0\le r<1} E(r,t).$$
		Therefore
		\begin{align*}
			\lim\limits_{t\rightarrow 1} E(r,t)\le\lim\limits_{t\rightarrow 1}\sup\limits_{0\le r<1} E(r,t)\le 11. 
		\end{align*}
		Next, we note that the quantity $\lim\limits_{t\rightarrow 1}\sup\limits_{0\le r<1} E(r,t)$ is constant that is independent of $r$ and $\lim\limits_{r\rightarrow 1}\lim\limits_{t\rightarrow 1} E(r,t)=11.$ Thus
		\begin{align}
			11=\lim\limits_{r\rightarrow 1}\lim\limits_{t\rightarrow 1} E(r,t)\le\lim\limits_{t\rightarrow 1}\sup\limits_{0\le r<1} E(r,t)\le 11
		\end{align}
		From which it follows that $\lim\limits_{t\rightarrow 1}\sup\limits_{0\le r<1} E(r,t)=11.$ Next, the relation \eqref{e5} implies that 
		$$\lim\limits_{t\rightarrow 1}N_t=11.$$
		and so from \eqref{log-000146} and \eqref{log-000150} it follows that 
		$$11=\lim\limits_{t\rightarrow 1}N_t\le\lim\limits_{t\rightarrow 1}\|P_{f_t}\|\leq 11.$$ 
		Thus finally we get $\lim\limits_{t\rightarrow 1}\|P_{f_t}\|=11.$ This completes the proof. 
	\end{proof}
	
	Our next result provides bound of $|f(z)|$ for a logharmonic mapping $f\in L_{\mathcal{R}}.$
	
	\begin{thm}\label{log-000200}
		Let $f\in L_{\mathcal{R}}$ be a sense-preserving logharmonic mapping with dilatation $\omega.$ Then for $|z|=r$ the sharp inequalities
		\begin{align}\label{log-000220}
			|f(z)|\le
			\begin{cases}
				\frac{e^{-r/\alpha(1+\alpha)}}{(1-r)^4} e^{(1/\alpha-\alpha)^2\log{(1+\alpha r)}} ,~~\text{when}~~|\omega(0)|=\alpha\neq 0 \\
				\frac{e^{-3r-r^2/2}}{(1-r)^4},\quad\quad\quad\quad\quad\quad\quad\quad\text{when}~~\omega(0)=0
			\end{cases}
		\end{align}
		hold.
	\end{thm}
	\begin{proof}
		A logharmonic mapping $f\in L_{\mathcal{R}}$ has the form $f(z)=e^{h(z)}\overline{G(z)},~~h\in\mathcal{R}$ with dilatation $\omega:\mathbb{D}\rightarrow\mathbb{D}$ such that
		$$\omega(z)=\frac{G'(z)}{G(z)h'(z)}.$$
		Now some easy manipulations provide that
		\begin{align}\label{log-000250}
			G(z)=e^{\int_{0}^{z}\omega(\zeta)h'(\zeta)d\zeta}.
		\end{align}
		Thus for $|z|=r<1$ we get
		\begin{align}\label{log-000280}
			|G(z)|\le e^{\int_{0}^{r}|\omega(\zeta)||h'(\zeta)|d|\zeta|}.
		\end{align}
	    It is well known that for analytic function $\omega:\mathbb{D}\rightarrow\mathbb{D}$ with $|\omega(0)|=\alpha$ the inequality
		\begin{align}|\label{log-000300}
			\omega(z)|\le \frac{\alpha+|z|}{1+\alpha|z|}
		\end{align}	
		and for analytic function $h\in\mathcal{R}$ the inequalities
		\begin{align}\label{log-000320}
			|h(z)|\le -|z|-2\log{(1-|z|)}\quad\text{and}\quad	|h'(z)|\le \frac{1+|z|}{1-|z|}
		\end{align}	
		hold. For $\alpha\neq 0$, using the bounds \eqref{log-000300} and \eqref{log-000320} in \eqref{log-000280}, we get
		\begin{align}\label{log-000340}
			|G(z)|\le & \exp\left({\int_{0}^{r}\frac{(\alpha+t)(1+t)}{(1+\alpha t)(1-t)}dt}\right)\\\nonumber
			= & \exp\left(\left(\frac{1-\alpha}{\alpha}\right)^2\log{(1+\alpha r)}-2\log{(1-r)}-\frac{r}{\alpha}\right)\\\nonumber
			= & \frac{e^{-r/\alpha}}{(1-r)^2}e^{(1/\alpha-\alpha)^2\log{(1+\alpha r)}}.
		\end{align}
	    Using \eqref{log-000320} and \eqref{log-000340}, we get
		$$|f(z)|\le e^{|h(z)|}|G(z)|\le \frac{e^{-r/\alpha(1+\alpha)}}{(1-r)^4}e^{(1/\alpha-\alpha)^2\log{(1+\alpha r)}}.$$
		
		On the other hand, when $\omega(0)=0,$ the inequality \eqref{log-000300} becomes $|\omega(z)|\le |z|$ and so \eqref{log-000340} becomes
		\begin{align*}
			|G(z)|\le & \exp\left({\int_{0}^{r}\frac{t(1+t)}{1-t}dt}\right)\\\nonumber
			= & \exp\left(\frac{-(4r+r^2)}{2}-2\log{(1-r)}\right)\\\nonumber
			= & \frac{e^{-2r-r^2/2}}{(1-r)^2}.
		\end{align*}
		and this, further, implies that
		$$|f(z)|\le e^{|h(z)|}|G(z)|\le \frac{e^{-3r-r^2/2}}{(1-r)^4}.$$
		
		To show that the estimates are sharp, we consider the logharmonic mapping $f(z)=e^{h(z)}\overline{G(z)}$ where $h(z)=-z-2\log{(1-z)}$ and dilatation $\omega(z).$ Then the first estimate is sharp for $\omega(z)=(\alpha+z)/(1+\alpha z)$ and the second estimate is sharp for $\omega(z)=z.$
	\end{proof}
	\begin{rem}\label{log-000370}
		From \eqref{log-000250}, we note that for every $f(z)=e^{h(z)}\overline{G(z)}\in L_{\mathcal{R}}$ the analytic function $G$ has the form $G=e^g$ where $g$ is analytic in $\mathbb{D}$ and having the form
		\begin{align*}
			g(z)=\sum\limits_{n=1}^\infty b_n z^n.
		\end{align*}
		Thus every $f\in L_{\mathcal{R}}$ can be expressed as $f=e^h\overline{e^g}$ where
		\begin{align}\label{log-000375}
			h(z)=z+\sum\limits_{n=2}^\infty a_n z^n\in\mathcal{R}\quad\text{and}\quad g(z)=\sum\limits_{n=1}^\infty b_n z^n.
		\end{align}
	\end{rem}

	Before going to our next result, we recall the definition of pre-Schwarzian norm of a harmonic mapping given by Hern{\'a}ndez and Mart{\'\i}n \cite{Hernandez-Martin-2015}. The pre-Schwarzian norm of a locally univalent harmonic mapping $f=h+\overline{g}$ with dilatation $\omega=g'/h'$ is defined by
	\begin{align*}
		||P_f|| = \sup_{z \in \mathbb{D}}(1-|z|^2)\left|\frac{h''}{h'}-\frac{\overline{\omega}\omega'}{1-|\omega|^2}\right|.
	\end{align*}
	The study of pre-Schwarzian norm has many application in geometric function theory to obtain different univalence criteria. In this regard, we collect the following result by Liu and Ponnusamy \cite{Liu-Ponnusamy--2018}.
	\begin{lem}\label{log-000720}
	Let $f= h+\overline{g}$ be a sense-preserving harmonic mapping in the unit disk $\mathbb{D}$. Then $\|P_f\|<\infty$ if and only if $f$ is uniformly locally univalent.
	\end{lem}
	We use this result to generate a family of uniformly locally univalent harmonic mappings.
	\begin{thm}\label{log-000380}
		Let  $f\in L_{\mathcal{R}}$ be sense-preserving logharmonic mapping in $\mathbb{D}.$ Then $\log{f}$ is locally uniformly univalent in $\mathbb{D}.$ 
	\end{thm}
	\begin{proof}
		Since $f\in L_{\mathcal{R}}$ has the form $f=e^h\overline{G}$ with $h\in\mathcal{R}$ and dilatation $\omega=G'/(Gh'),$ it follows that $\log{f}=h+\overline{\log{G}}=F$ is harmonic mapping having same dilatation $\omega.$ Now, from \eqref{log-000110} and the Schwarz Pick lemma, the pre-Schwarzian norm of $F$ is
		\begin{align*}
			\|P_F\|= & \sup\limits_{z\in\mathbb{D}}(1-|z|^2)|P_F(z)|\\
			= & \sup\limits_{z\in\mathbb{D}}(1-|z|^2)\left|P_h(z)-\frac{\omega'(z)\overline{\omega(z)}}{1-|\omega(z)|^2}\right|\\
			\le & \sup\limits_{z\in\mathbb{D}}(1-|z|^2)\left(\frac{2}{1-|z|^2}+\frac{|z|}{1-|z|^2}\right)\\
			= & 3.
		\end{align*}
		So $F=\log{f}$ is uniformly locally univalent by Lemma \ref{log-000720}.
	\end{proof}
	
	We have seen that functions $f\in L_{\mathcal{R}}$ are non vanishing in $\mathbb{D}$. Here, we use a transformation 
	$$f\rightarrow zf$$ 
	for $f\in L_{\mathcal{R}}$ to generate a family $L^0_{\mathcal{R}}$ of mappings $zf\in C^1(\mathbb{D})$ that fixes origin. Thus every $f\in L^0_{\mathcal{R}}$ has the form $f=zF(z)$ with $F=e^h\overline{e^g}\in L_{\mathcal{R}}$ and dilatation of $f$ is
	\begin{align}\label{log-000400}
		\omega(z)=\frac{zg'(z)}{(1+zh'(z))}
	\end{align}
	where the analytic functions $h,~g$ are given by \eqref{log-000375}. It is important to note that the dilatation $\omega$ of $f\in L^0_{\mathcal{R}}$ follows the property that $\omega(0)=0.$ Now, we are in a position to discuss starlikeness of mappings in the class $L^0_{\mathcal{R}}$. A function $f\in L^0_{\mathcal{R}}$ is called starlike in $\mathbb{D}$ if the image domain $f(\mathbb{D})$ is starlike with respect to origin. A function $f\in L^0_{\mathcal{R}}$ is called fully starlike if $f(\mathbb{D}_r)$ is starlike for each $r\in(0,1)$ where $\mathbb{D}_r=\{z\in\mathbb{C}:|z|<r\}.$ A function is hereditarily starlike if it is fully starlike and univalent. Putting $\lambda=0$ in \cite[Lemma 2.1]{Ma-Ponnusamy-Sugawa-2022}, we get the following definition of a hereditarily starlike function.
	\begin{definition}
		Suppose that a function $f\in C^1(\mathbb{D})$ satisfies the condition that $f(z)=0$ if and only if $z=0,$ and that $J_f=|f_z|^2-|f_{\overline{z}}|^2>0$ on $\mathbb{D}.$ Then $f$ is one-one in $\mathbb{D}$ and $f(\mathbb{D}_r)$ is starlike for each $0<r<1$ if and only if
		\begin{align*}
			{\rm Re\,}\left(\frac{Df(z)}{f(z)}\right)>0,\quad z\in\mathbb{D}\setminus\{0\}.
		\end{align*}
	\end{definition}

	\begin{example}
		There are logharmonic mappings $f\in L^0_{\mathcal{R}}$ which are not hereditarily starlike in the unit disk $\mathbb{D}.$
		
		Let us consider a logharmonic mapping $f(z)=ze^{h(z)}\overline{e^{g(z)}}$ where $h(z)=\log{(1+z)}$ and the dilatation $\omega(z)=-z.$ Here, $h'(z)=1/(1+z)$ and so $h(0)=h'(0)-1=0$, and ${\rm Re\,}h'(z)$ is positive in the unit disk $\mathbb{D}$. This implies that $f\in L^0_{\mathcal{R}}.$ Next, from \eqref{log-000400} it is follows that 
		\begin{align*}
			{\rm Re\,}\frac{Df(z)}{f(z)}= {\rm Re\,}\left[(1-\omega(z))(1+zh'(z))\right]={\rm Re\,}(1+2z)                      
		\end{align*}
		which is non positive in $\mathbb{D}\cap\{z\in\mathbb{C}:-1<{\rm Re\,}z<-1/2\}.$
	\end{example}
	Our next result provides sufficient condition under which a logharmonic mapping $f\in L^0_{\mathcal{R}}$ is hereditarily starlike.
	\begin{thm}\label{p6-100650}
		Let $f=ze^{h(z)}\overline{e^{g(z)}}\in L^0_{\mathcal{R}}$ be a sense-preserving logharmonic mapping in $\mathbb{D}$ of the form \eqref{log-000375} such that 
		\begin{align}\label{p6-100700}
			|1-b_1|+\sum\limits_{n=2}^\infty n|a_n-b_n|\leq 1.
		\end{align}
		Then $f$ is fully starlike in $\mathbb{D}.$
	\end{thm}
	\begin{proof}
		For $z\in\mathbb{D}\setminus\{0\},$ we have 
		\begin{align*}
			{\rm Re\,}\left(\frac{Df(z)}{f(z)}\right)=&{\rm Re\,}\left(1+zh'(z)-\overline{zg'(z)}\right)\\
			= & 1+{\rm Re\,}\left(zh'(z)-zg'(z)\right)\\
			\ge & 1-\left|zh'(z)-zg'(z)\right|\\
			=& 1-\left|(1-b_1)z+\sum\limits_{n=2}^\infty n(a_n-b_n)z^n\right|\\
			\ge & 1-|1-b_1||z|-\sum\limits_{n=2}^\infty n\left|a_n-b_n\right||z|^n\\
			=&1-|z|\left(|1-b_1|+\sum\limits_{n=2}^\infty n\left|a_n-b_n\right||z|^{n-1}\right)\\
			> & 1-\left(|1-b_1|+\sum\limits_{n=2}^\infty n\left|a_n-b_n\right|\right).
		\end{align*}
		Thus ${\rm Re\,}\left(\frac{Df(z)}{f(z)}\right)>0$ if \eqref{p6-100700} holds and so $f$ is hereditarily starlike.
	\end{proof}
	Theorem \ref{p6-100650} can be used to construct logharmonic hereditarily starlike functions in the class $L^0_{\mathcal{R}}.$ For $\alpha\in[0,1]$ let us consider logharmonic mappings $f_\alpha(z)=ze^{h(z)}\overline{e^{g(z)}}\in L^0_{\mathcal{R}}$ with $h(z)=z+\alpha z^2/2$ and dilatation $\omega(z)=z.$ Here, we find that $g(z)=z+z^2/2+\alpha z^3/3.$ Next, it is easy to verify that the condition \eqref{p6-100700} hold for every $\alpha\in[0,1]$ from which it follows that $f_\alpha$ are hereditarily starlike. Here, we present images of $\mathbb{D}$ under $f_\alpha$ for different values of $\alpha$, (generated using wolfram mathematica).
		\begin{figure}[H]
				\subfloat[$\alpha=0.2$\label{sp3}]{%
						\includegraphics[width=0.43\textwidth, height=7cm]
						{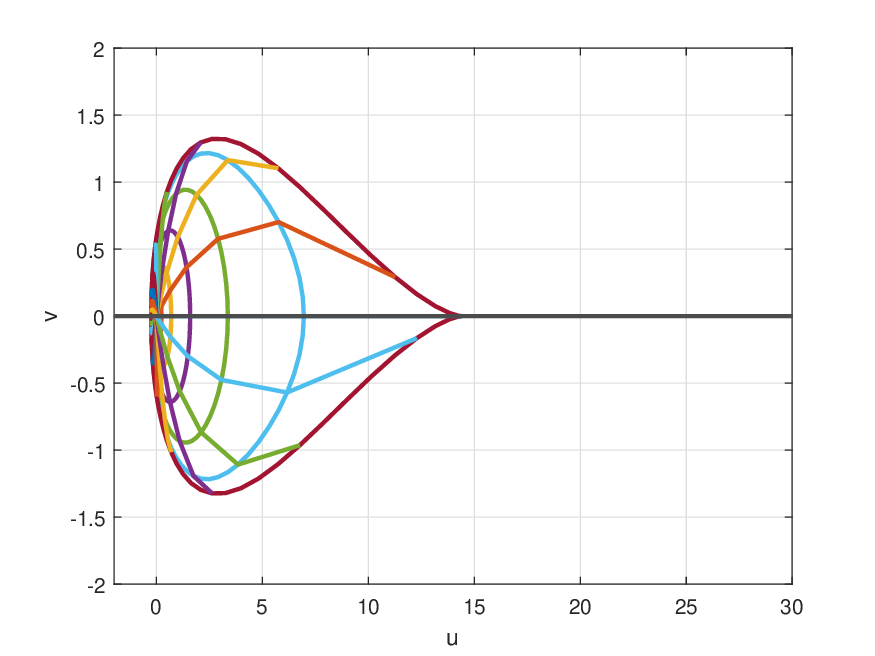}}
				\hspace{.2cm}
				\subfloat[$\alpha=0.6$\label{sp5}]{%
						\includegraphics[width=0.43\textwidth, height=7cm]
						{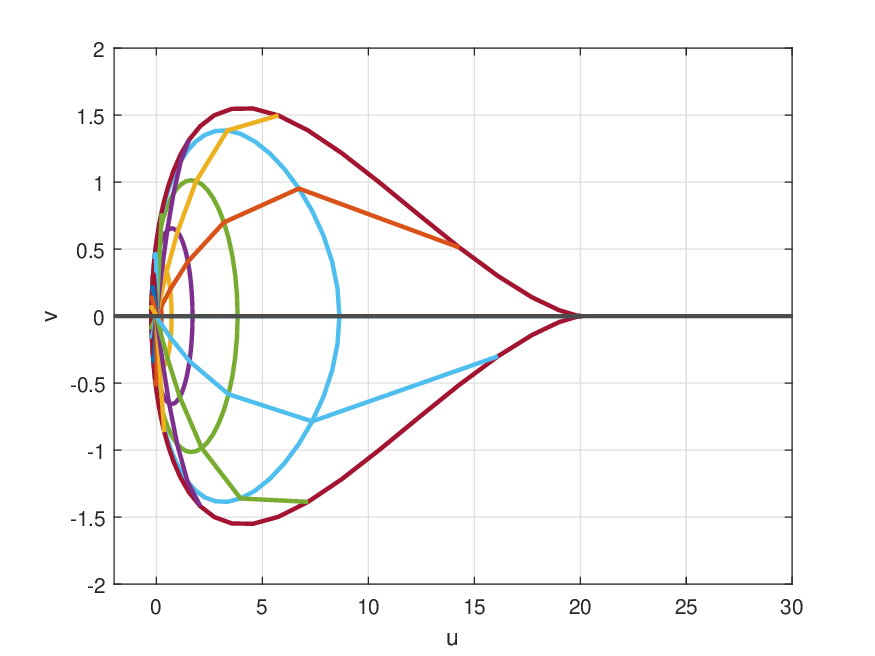}}
				\\
				\vspace{.2cm}
				\subfloat[$\alpha=0.8$\label{sp8}]{%
						\includegraphics[width=0.43\textwidth, height=7cm]
						{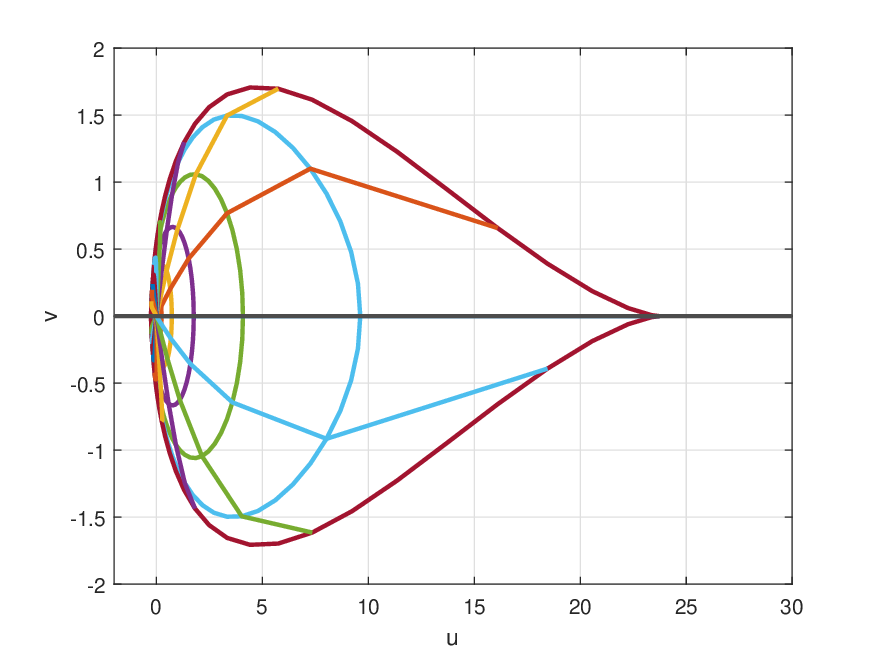}}
				\hspace{.2cm}
				\subfloat[$\alpha=1$\label{sp0}]{%
						\includegraphics[width=0.43\textwidth, height=7cm]
						{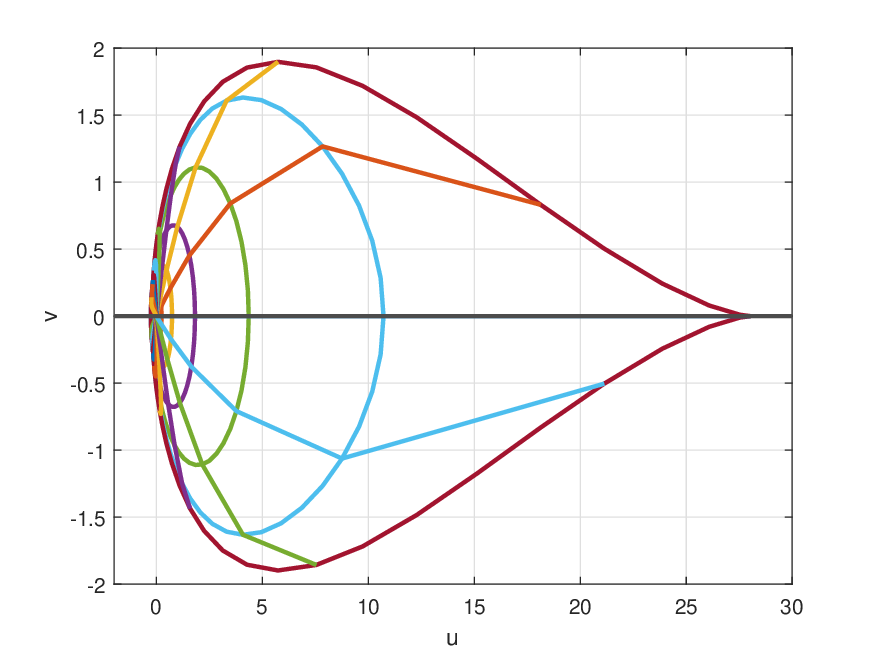}}
				\caption{$f_\alpha(\mathbb{D})$ for certain values of $\alpha.$}
				\label{fig-1}
			\end{figure}

	\section*{Declarations:}

	\noindent\textbf{Data availability:}
	Data sharing not applicable to this article as no data sets were generated or analyzed during the current study.\\

	\noindent\textbf{Acknowledgement:}
	The author thanks the Department of Science and Technology, Ministry of Science and Technology, Government of India
	for the financial support through DST-INSPIRE Fellowship (No. DST/INSPIRE Fellowship/2018/IF180967).\\
	
	\noindent\textbf{Conflict of interest:} The author declares that he has no conflict of interest.

\end{document}